\documentclass[a4paper,oneside,11pt]{article}

\usepackage{amsxtra,amssymb,amsmath,amsthm,amsbsy,enumerate}
\usepackage[all,2cell,v2]{xy}\UseTwocells\UseHalfTwocells
\usepackage{graphicx}

\usepackage[dvipsnames]{xcolor}

\theoremstyle{definition}
\newtheorem*{thm}{Theorem}
\newtheorem{theorem}{Theorem}[section]
\newtheorem{lemma}[theorem]{Lemma}
\newtheorem{remark}[theorem]{Remark}
\newtheorem{corollary}[theorem]{Corollary}
\newtheorem{example}[theorem]{Example}
\newtheorem{proposition}[theorem]{Proposition}

\def\R{{\mathbb R}}
\def\U{{\mathcal U}}

\def\hol{{\rm Hol}}

\def\id{{\rm id}}

\def\xto{\xrightarrow}
\def\to{\rightarrow}

\def\toto{\rightrightarrows}
\def\xfrom{\xleftarrow}

\def\action{\curvearrowright}
\def\xto{\xrightarrow}
\def\to{\rightarrow}

\def\toto{\rightrightarrows}
\def\xfrom{\xleftarrow}

\def\action{\curvearrowright}

\def\r#1{|_{#1}}
\def\Z{\mathbb Z}

\newcommand{\n}[1]{\left\lVert#1\right\rVert}
\newcommand{\m}[1]{\left\lvert#1\right\rvert}
\newcommand{\p}[1]{\left\langle#1\right\rangle}

\def\o{\overline}

\title{\textbf{Geodesics on Riemannian stacks}}
\author{Matias del Hoyo \and Mateus de Melo}
\date{}

\begin{document}

\maketitle

\begin{abstract}
Metrics on Lie groupoids and differentiable stacks have been introduced recently, extending the Riemannian geometry of manifolds and orbifolds to more general singular spaces. 
Here we continue that theory, studying stacky curves on Riemannian stacks, measuring their length using stacky metrics, and introducing stacky geodesics. Our main results show that the length of stacky curves measure distances on the orbit space, 
characterize stacky geodesics as locally minimizing curves, and 
establish a stacky version of Hopf-Rinow Theorem.
We include a concise overview that bypasses nonessential technicalities, and we lay stress on the examples of orbit spaces of isometric actions and leaf spaces of Riemannian foliations.  
\end{abstract}


\setcounter{tocdepth}{1}
\tableofcontents

\section{Introduction}


Differentiable manifolds play a central role in nowadays mathematics,
serving as models for spaces in geometry, topology, analysis and mathematical physics. While manifolds are homogeneous, some situations demand to deal with spaces with singularities, and to perform differential geometry over them. A framework that generalizes the notion of manifold and which has received much attention lately is that of Lie groupoids and differentiable stacks. 


Lie groupoids are a categorification of the notion of manifold. They permit a unified treatment to classical geometries such as actions, fibrations and foliations \cite{mkbook,mmbook}. Also, they serve as geometric models for noncommutative algebras, and play a role in desingularizing Poisson and Dirac structures \cite{dswbook}. 
Morita equivalences of Lie groupoids provide a working definition for differentiable stacks, avoiding the categorical apparatus of the original definition.


Stacks are sheaves of groupoids, originally introduced in algebraic geometry, which make sense in very general contexts, and have shown to be useful to model singular quotients and moduli spaces. Differentiable stacks \cite{bx,dh}, the incarnation of that theory in differential geometry,  include manifolds and orbifolds as particular examples, and more general singular spaces such as orbit spaces of actions and leaf spaces of foliations. These are our main examples.


While several tensors admit simple groupoid versions, such as symplectic forms and vector fields \cite{mkbook}, the case of metrics turns out to be more subtle. After several attempts, a general theory for Riemannian groupoids was proposed in \cite{dhf1}, based on the notion of Riemannian submersion, and inspired in a simplicial approach to groupoids via their nerve. These metrics are Morita invariant, hence inducing a notion of metric on differentiable stacks \cite{dhf2}.


Proper groupoids have compact isotropies and Hausdorff orbit spaces, and they admit averaging systems, which can be used to construct Riemannian metrics. The stacks arising from proper groupoids are called separated.
Proper Lie groupoids are linearizable around their orbits by the exponential maps of a metric. It follows that separated stacks are locally modeled by linear representations of compact groups. 





\medskip



In this paper we study curves, and specially geodesics, on Riemannian stacks. 
This general framework allows us to unify and generalize aspects of previous works on the Riemannian geometry of orbit spaces \cite{aalm,gl} and leaf spaces \cite{abt,salem}. 
We establish stacky version of some fundamental results for Riemannian manifolds and orbifolds, and explain why some other cease to hold in the new broad context. 


Following \cite{dhf1,dhf2}, our object of study is a Lie groupoid $G\toto M$ equipped with compatible metrics $\eta^{(0)},\eta^{(1)},\eta^{(2)}$ on objects, arrows and pair of composable arrows. We will recall the precise definition in next section.
The induced distance in $M$ yields a pseudo-distance $d_N$ in the coarse orbit space $M/G$, which was studied in \cite{ppt}. 
Here we think of $(G\toto M,\eta)$ as a device to perform Riemannian geometry on the distance space $(M/G,d_N)$.


Given $G\toto M$ a Lie groupoid, we denote by $[M/G]$ its orbit stack, and we model a {\bf stacky curve} $\alpha:I\to[M/G]$ by a cocycle $(a_{ji})$, which
is roughly a sequence of curves $a_{ii}$ in $M$ linked on the extremes by arrows $a_{ji}$ in $G$. More precisely, a cocycle is a groupoid map from the groupoid arising from a cover of the interval. We will make these definition precise, interpret them in the fundamental examples, and relate them with ad hoc definitions in the literature, as  \cite[3.3]{mm2}.
$$(a_{ji}):(\coprod U_{ji}\toto\coprod U_i)\to (G\toto M)$$


Given $x\in M$, the orbit $O_x$ represents a point on the stack $[M/G]$. We use the isotropy representation $G_x\action N_xO_x=T_xM/TO_x$ to model the tangent space to $[M/G]$ at $O_x$.  
A groupoid metric $\eta$ on $G\toto M$ yields an equivariant inner product on  $N_xO_x$, and we can define the {\bf normal norm} $\n{v}_N$ of any vector $v\in T_xM$.
If $\alpha:I\to[M/G]$ is a stacky curve represented by a good cocycle $(a_{ji})$, namely one defined over a dimension 1 cover, then the {\bf length} of $\alpha$ is
$$\ell(\alpha)=
\int_{I} \n{\alpha'(t)} dt=
\sum_{i}\int_{U_i}\n{a'_{i}(t)}_Ndt - 
\sum_{i}\int_{U_{i+1,i}}\n{a'_{i+1,i}(t)}_Ndt.$$

Our first main theorem shows that the length of stacky curves recovers the normal pseudo-distance $d_N$ on $M/G$, which was studied in \cite{ppt} and that we will review in section 3.3.

\begin{thm}[\ref{thm:distance-length}]
If $(G\toto M,\eta)$ is a Riemannian groupoid and $M/G$ is connected, 
then $d_N([x],[y])$ is the infimum of the lengths $\ell(\alpha)$ of stacky curves $\alpha:I\to[M/G]$ connecting $[x]$ and $[y]$.
\end{thm}

Two groupoid metrics $\eta_1$ and $\eta_2$ on $G\toto M$ are said to be equivalent if they induce the same inner product on the normal directions $N_xO$ for every $x$. Equivalence classes of metrics are a Morita invariant, hence inducing a notion of metric on the orbit stack (cf. \cite{dhf2}). As immediate corollaries of Theorem \ref{thm:distance-length} we see that the distance $d_N$ on $M/G$ depends only on the class of the metric, and that $d_N$ is preserved by Riemannian Morita equivalences.


Then we introduce a working definition for geodesics on stacks. Given $(G\toto M,\eta)$ a Riemannian groupoid, we say that a stacky curve $\alpha:I\to[M/G]$ is a {\bf geodesic} if it can be presented by a cocycle $(a_{ji})$ on which the curves $a_i=a_{ii}:U_i\to M$ and $a_{ji}:U_{ji}\to G$ are geodesics orthogonal to the orbit foliations. We interpret this definition on the fundamental examples, and establish existence, uniqueness and other basic properties.


For proper Riemannian groupoids, we manage to prove a stacky version of Gauss lemma, asserting that if $v\in T_pM$ is small then $d_N([\exp(v)],[x])=\n{v}_N$. From this, we derive our second theorem, characterizing stacky geodesics as locally minimizing curves. We say that a stacky curve $\alpha:I\to[M/G]$ is {\bf minimizing} at $t_0\in I$ if the length of $\alpha$ from $t_0$ to $t$ equals the normal distance between $\alpha(t)$ and $\alpha(t_0)$ for every $t$ near enough $t_0$.

\begin{thm}[\ref{thm:geodesics-distance}]
Given $(G\toto M,\eta)$ a proper Riemannian groupoid and $\alpha:I\to[M/G]$ a stacky curve, then $\alpha$ is a geodesic if and only if it is minimizing at every $t_0\in I$.
\end{thm}


Note that a stacky geodesic does not minimize distance between any pair of nearby points, even in the orbifold case. For instance, on the plane modulo a reflection, a straight line does not minimize the distance between points on different side of the axis. 
We will see in Proposition \ref{prop:stratum} that if in a separated stack a geodesic minimizes the distance between two given points then the intermediate points can not belong to a more singular stratum.

A remarkable corollary of Theorem \ref{thm:geodesics-distance} is that, when working with proper groupoids, equivalent metrics give rise to the same geodesics, for they only depend on the distance. Then geodesics are an intrinsic notion for  separated Riemannian stacks. We conjecture this is true for arbitrary Riemannian stacks.


Finally we study completeness. Given $(G\toto M,\eta)$ a proper Riemannian groupoid,  
one of the main theorems of \cite{ppt} shows that if the metric $\eta^{(0)}$ is geodesically complete then the distance $d_N$ on $M/G$ is complete. Our framework allows us to strengthen this result, achieving the following stacky version of Hopf-Rinow theorem. We say that a Riemannian stack is {\bf geodesically complete} if stacky geodesics can be extended to the whole real line.


\begin{thm}[\ref{thm:hopf-rinow}]
A separated Riemannian stack $([M/G],[\eta])$ is geodesically complete if and only if the coarse orbit space $(M/G,d_{N})$ is a complete metric space.
\end{thm}


A nice corollary is that if the coarse orbit space $M/G$ is compact then $([M/G],[\eta])$ is geodesically complete. To relate our result with that of \cite{ppt}, note that if $(M,\eta^{(0)})$ is geodesically complete then $[M/G]$ also is, for the projection $M\to [M/G]$ is a stacky Riemannian submersion, in the sense that geodesics on $M$ orthogonal to the fibers project to geodesics on $[M/G]$. This enhances the claim that $M\to M/G$ is a submetry \cite{ppt}.


\medskip

Most of our results are stated for proper Riemannian groupoids, but they can easily be extended to Riemannian groupoids that are linearizable and have Hausdorff orbit spaces. In fact, these two conditions readily imply that the corresponding stack is locally a product of a separated stack, coming from the effective part of the normal representation, and the classifying stack of a group, the ineffective part, with no transverse information.

\medskip

{\bf Organization.} 
In $\S 2$ we provide a self-contained presentation for Lie groupoids, differentiable stacks and Riemannian structures over them. 
In $\S 3$ we study stacky curves and use them to measure distances.
In $\S 4$ we introduce geodesics and develop their theory.

\medskip

{\bf Acknowledgements.} 
The results are part of the second author's Ph.D. Thesis at IMPA, under the advice of H. Bursztyn. We thank him for his personal and academic guidance and support all along the process.
We also thank M. Alexandrino, R. Fernandes, and I. Struchiner for enriching talks that shaped our view of the topic. MdM thanks R. Fernandes and the Department of Mathematics at UIUC for the hospitality during Fall 2017.
MdH was partially supported by National Council for Scientific and Technological Development - CNPq grant 303034/2017-3.
MdM was partially supported by CNPq Ph.D. Fellowship 140771/2015-8, and by CAPES exchange student Fellowship 88881.135952/2016-01.

\section{Background: groupoids, stacks and metrics}


We review the main concepts needed in the rest of the paper, to set notations and conventions, and to serve as a quick reference for the reader. For a thorough exposition on Lie groupoids, we suggest \cite{mkbook,mmbook}. We will use groupoids as models for differentiable stacks, as explained for instance in  \cite{dh}. Riemannian groupoids were introduced in \cite{dhf1}, and the Morita invariance of metrics and a definition of Riemannian stack was presented in \cite{dhf2}.


\subsection{Lie Groupoids}


A {\bf Lie groupoid} $G\toto M$ consists of a manifold $M$ of objects and a manifold $G$ of arrows, two surjective submersions $s,t:G\to M$ indicating the source and target of arrows, and a smooth associative multiplication $m:G\times_M G\to G$ defined over composable pairs,  admitting unit $u: M\to G$ and inverse $i:G\to G$ maps, subject to the usual groupoid axioms. 


Given $x\in M$, its {\bf isotropy group} $G_{x}=s^{-1}(x)\cap t^{-1}(x)$ is a Lie group and an embedded submanifold of $G$, and its {\bf orbit} 
$O_x=\{y\in M:\exists y\xfrom g x\}$ is a submanifold of $M$, possibly not embedded. 
The components of the orbits define the {\bf characteristic foliation} $F_M$ on $M$, that may be singular. 
The {\bf coarse orbit space} $M/G$ is endowed with the quotient topology. 

\begin{example}\label{ex:submersion1}
A surjective submersion $p:M\to B$ gives rise to a {\bf submersion groupoid}, where $G=M\times_{B}M$, $s$ and $t$ are the projections, and $m$ is given by $(z,y)\cdot(y,x)=(z,x)$. The orbits are the fibers of $p$, the isotropies are trivial, and $M/G\cong B$.
As a subexample, if $(U_i)_i$ is an open cover of $M$ then the submersion $\coprod U_i\to M$ leads to the {\bf \v{C}ech groupoid} 
$$G_{\U}=\left(\coprod U_j\cap U_i\toto \coprod U_i\right)$$
When a Lie groupoid $G\toto M$ has no isotropy and its orbit space $M/G$ is Hausdorff then $B=M/G$ is naturally a manifold, and $G\toto M$ identifies with a submersion groupoid \cite{dh}.
\end{example}


\begin{example}\label{ex:orbifold1}
An orbifold is classically defined as a Hausdorff space $X$ locally modeled by a finite group acting on an Euclidean open. An {\bf orbifold groupoid} is built out of an orbifold atlas $(G_i\action U_i,\phi_i)_i$, $\phi_i:U_i/G_i\to X$, by emulating the \v{C}ech groupoid from previous example. 
The objects are $M=\coprod U_i$, and the arrows $G$ consist of compositions of germs of maps in some $G_i$, endowed with a sheaf-like manifold structure  \cite{mmbook}. The orbit space is $M/G=X$.
\end{example}


\begin{example}\label{ex:action1}
A smooth action $K\action M$ of a Lie group $K$ on a manifold $M$ gives rise to an {\bf action groupoid}  $G \toto M$, with 
$G=K\times M$, $s(g, x) = x$, $t(g, x) = g\cdot x$ and $(h, y)(g, x) = (hg, x)$. Note that orbits and isotropy groups agree with the usual ones for the action. In particular, the coarse orbit space $M/G$ agrees with the usual orbit space $M/K$.
\end{example}


\begin{example}\label{ex:foliation1}
A regular foliation $F$ on a manifold $M$ gives rise to a {\bf monodromy groupoid} $\pi_1(F)\toto M$, whose arrows are leafwise homotopy classes of paths. Its orbits are exactly the leaves of $F$ and the isotropy groups are their fundamental groups. Each arrow $[\gamma]\in\pi_1(F)$ induces 
the germ of a transverse diffeomorphism, the holonomy of the path, and the quotient of $\pi_1(F)\toto M$ by holonomy classes is the {\bf holonomy groupoid} $\hol(F)\toto M$ \cite{mmbook}. These Lie groupoids may have non-Hausdorff manifold of arrows.
\end{example}


Given $G\toto M$ a Lie groupoid, the {\bf normal representation} of the isotropy $G_x$ over  $N_xO=T_xM/T_xO$ is given by $g\cdot[v]=[d_gt(w)]$, where $w$ is any vector such that $d_gs(w)=v$. 
If $x$ and $y$ are in the same orbit, then the isotropy groups $G_x,G_y$ are conjugated by any arrow \smash{$y\xfrom g x$}, and this isomorphism is compatible with the normal representations. In fact, we can organize the pointwise normal representations and conjugations into a groupoid representation of the restriction to the orbit $(G_O\toto O)\action(NO\to O)$ \cite{dh}, that serves as a {\bf linear local model}.


A {\bf map} between Lie groupoids $f:(H\toto N)\to(G\toto M)$ is a smooth functor, given by smooth maps $f^{(0)}:N\to M$ on objects and $f^{(1)}:H \to G$ on arrows, which together preserve the groupoid structure. Such a map yield morphisms between the isotropy groups $f_x:H_x\to G_{f(x)}$, and since it must send orbits to orbits, it also yields a continuous map between the orbit spaces $[f]:N/H\to M/G$, as well as linear maps on the normal directions $[d_xf]:N_xO\to N_{f(x)}O$.

A {\bf natural isomorphism} $f\cong f'$ between two maps is given by a smooth map $h: N\to G$ with $s\circ h=f$ and $t\circ h=f'$ and satisfying $f'(g)h_x=h_y f(g)$ for all \smash{$y\xfrom g x$} in $G$.
If $f,f'$ are isomorphic then $[f]=[f']:N/H\to M/G$ and for every $x\in N$ the maps $f_x,[d_xf]$ are related to $f'_x,[d_x f']$ by conjugation by $h_x$.


Given a Lie groupoid $G\toto M$ and an orbit $O\subset M$, and writing $G_O$ for the submanifold of arrows with source and target in $O$, the Lie groupoid $NG_O\toto NO$ whose objects and arrows are the normal bundles of $O$ and $G_O$, and whose structure maps are induced by differentiation, encodes the normal representation discussed before. Then $G\toto M$ is {\bf linearizable} around $O$ if there are opens $O\subset U\subset NO$ and $O\subset V\subset M$ and a Lie groupoid isomorphism
$$f:((NG_O)_U\toto U) \cong (G_V\toto V).$$


A Lie groupoid $G\toto M$ is {\bf proper} if $G$ is Hausdorff and the source-target map $G\to M\times M$ is proper. 
In such a groupoid the isotropy groups $G_x$ are compact, the orbits $O\subset M$ are closed embedded, and the orbit space $M/G$ is Hausdorff. A fundamental theorem by A. Weinstein \cite{w} and N. Zung \cite{z} shows that a proper Lie groupoid is linearizable around its orbits. A novel approach to linearization of groupoids using Riemannian metrics was introduced in \cite{dhf1}.
%
%
%
%
%
%

\subsection{Differentiable stacks}




Following \cite{dh,dhf2}, we say that a map $f:(H\toto N)\to( G\toto M)$ between Lie groupoids is {\bf Morita} if it preserves the transverse geometry, in the sense that the following hold:
\begin{itemize}\itemsep=2pt
 \item $[f]:N/H\to M/G$ is an homeomorphism;
 \item $f_x:H_x\to G_{f(x)}$ is an isomorphism for all $x$; and
 \item $[d_xf]:N_x O\to N_{f(x)} O$ is an isomorphism for all $x$.
\end{itemize}
A {\bf differentiable stack} is a Lie groupoid up to Morita equivalence, in the sense that $H$ and $G$ define the same stack if there is a third groupoid $\tilde H$ and Morita maps $\tilde H\to H$ and $\tilde H\to G$.
We write $[M/G]$ for the differentiable stack presented by $G\toto M$, and think of it as the space $M/G$ enriched with extra smooth info describing the transverse geometry. 
A stack $[M/G]$ is {\bf separated} if it is presented by a proper Lie groupoid $G\toto M$. 
Because of the linearization theorem, separated stacks can be locally recovered from their normal representations.

\begin{example}\label{ex:submersion2}
Smooth manifolds can be identified with separated differentiable stacks $[M/G]$ with no isotropy \cite{dh}. In fact, a Lie groupoid $G\toto M$ is Morita equivalent to the unit groupoid of a manifold $B\toto B$ if and only if $G\toto M$ is a submersion groupoid, as defined in Example \ref{ex:submersion1}.
\end{example}

\begin{example}\label{ex:orbifold2}
We sketched in Example \ref{ex:orbifold1} how to get a Lie groupoid out of an orbifold chart. Different charts yield different Lie groupoids, but a common refinement leads to Morita maps, so all these groupoids are Morita equivalent. 
From a modern perspective, orbifolds are defined as separated stacks with finite isotropy groups. This greatly simplifies the study of maps and suborbifolds. Details on the correspondence between the classic and new approach are in \cite{lerman,moerdijk,mmbook}.
\end{example}

\begin{example}\label{ex:action2}
Given $K\action M$ a Lie group acting over a manifold, the differentiable stack $[M/G]$ arising from the action groupoid (cf. Example \ref{ex:action1}) encodes the equivariant geometry of the action. When the action is free and proper this is just the quotient manifold. In general, $[M/G]$ can be used to compute the equivariant cohomology, as a finite-dimensional variant for the well-known Borel's recipe. For $K$ discrete, the stack $[M/G]$ is studied in \cite{cdhp}.
\end{example}

\begin{example}\label{ex:foliation2}
If $F$ is a foliation on $M$, then the monodromy and holonomy groupoids $\pi_1(F)\toto M$ and $\hol(F)\toto M$ defined in Example \ref{ex:foliation1} are not Morita equivalent in general, there may be nontrivial loops with trivial holonomy. The projection still preserves much of the transverse geometry.
$$(\pi_1(F)\toto M)\to(\hol(F)\toto M)$$
It is a homeomorphism in the orbit spaces and an isomorphism on the normal directions, so they should be equivalent in a theory of reduced differentiable stacks, as the one proposed in \cite{tr}.
\end{example}


To define curves and more general maps between differentiable stacks we use a cocycle formulation.
%
Given $H\toto N$ a Lie groupoid and $\U=(U_i)_i$ an open cover of $N$, the {\bf \v{C}ech groupoid} $H_\U$ is the pullback of $H$ along the submersion $\coprod U_i \to N$ (cf. Example \ref{ex:submersion1}). This groupoid has an explicit construction
   
$$H_\U=\bigg(\coprod_{j,i} H(U_j,U_i)\toto\coprod_{i}U_i\bigg)$$
with arrows \smash{$(y,j)\xfrom{(g,j,i)}(x,i)$} for $x\in U_i$, $y\in U_j$ and \smash{$y\xfrom g x$} in $H$, and product $(h,k,j)(g,j,i)=(hg,k,i)$.
The obvious projection $\pi_\U:H_\U\to H$ is Morita. Note that if $\U'$ is a refinement of $\U$ then we can factor  $\pi_{\U'}$ via $\pi_\U$, and the resulting map $H_{\U'}\to H_{\U}$ is unique up to isomorphism. 


A {\bf cocycle} over $H\toto M$ with values in $G\toto M$ is a Lie groupoid map $f:H_\U\to G$ for some $\U=(U_i)_i$ open cover of $N$. Two cocycles $(\U,f)$ and $(\U',f')$ are {\bf equivalent} if there is a common refinement $\U''$ of $\U$ and $\U'$ and an isomorphism between the restricted groupoid maps $f\r{H_{\U''}}\cong f'\r{H_{\U''}}:H_{\U''}\to G$. 
A {\bf stacky map} $\phi:[N/H]\to [M/G]$ is a cocycle modulo equivalences. They form a category, 
whose isomorphisms are the Morita equivalences \cite[4.5.4]{dh}. 

\begin{example}\label{ex:bundles}
Viewing a manifold $M$ as a groupoid with only identity arrows and a Lie group $K$ as a groupoid with a single object, a groupoid cocycle $M_\U\to K$ recovers the usual notion of cocycle, and a stacky map 
$M=[M/M]\to [\ast/K]$
is the same as a principal $K$-bundle over $M$ \cite{bx,dh}.
Thus $[\ast/K]$ is a finite dimensional model for the classifying space of $K$.
\end{example}


Given $G\toto M$ a Lie groupoid, $x\in M$ and $O\subset M$ its orbit, the action groupoid of the normal representation $G_x\action N_xO$ canonically sits inside the local model and the inclusion is Morita.
$$(G_x\times N_xO\toto N_xO)\to (NG_O\toto NO)$$
We refer to the coarse orbit spaces of these groupoids $N_xM/G_x\cong NO/NG_O$ as the {\bf coarse tangent space} and denote it by $T_{[x]}[M/G]$. 
If $\phi:[N/H]\to [M/G]$ is a stacky map represented by a cocycle $(\U,f)$, and if $y\in U_i$, the {\bf coarse differential} $d_{[y]}\phi:T_{[y]}[N/H]\to T_{[f(y)]}[M/G]$ is given by $d_{[y]}\phi([v])=[d_{(y,i)}f]([v])$.
This is well-defined, in the sense that it only depends on the isomorphism class of the cocycle.

\begin{remark}
Alternatively, Morita maps can be defined as maps that are smoothly fully faithful and essentially surjective \cite{mmbook}. And there is yet another approach to Morita equivalences, by means of principal bibundles \cite{hk}. The equivalence between both definitions of Morita maps and the correspondence with bibundles can be found in \cite{dh}.
\end{remark}


\subsection{Riemannian groupoids and stacks}


In a Lie groupoid $G\toto M$ the pairs of composable arrows $G^{(2)}=G\times_M G$ can be identified with the space of commutative triangles, so it gains a natural $S_3$-action by permuting the vertices.
A {\bf Riemannian groupoid} $(G\toto M,\eta)$ is a Lie groupoid with an $S_3$-invariant metric $\eta=\eta^{(2)}$ on $G^{(2)}$ which is transverse to the multiplication $m:G^{(2)}\to G$ \cite{dhf1}. 
Such an $\eta$ induces metrics $\eta^{(1)}$ and $\eta^{(0)}$ on $G$ and $M$ satisfying the following basic properties:

\begin{itemize}
 \item the maps $m,\pi_1,\pi_2:G^{(2)}\to G$ and $s,t:G\to M$ are Riemannian submersions;
 \item the foliations $F_M$ and its pullback $F_G=s^*F_M=t^*F_M$ on $G$ are Riemannian (cf. Example \ref{ex:foliation3}); 
 \item the normal representations $G_O\action NO$ are by isometries;
 \item the units $u(M)\subset G$ form a totally geodesic submanifold.
\end{itemize}


This definition corrects and extends several previous approaches, and it has two fundamental features. Firstly, it admits plenty of examples, in particular, every proper Lie groupoid can be endowed with one of these metrics via an averaging argument, similar to the construction of metrics on manifolds using partitions of 1. Secondly, the exponential maps of the metrics can be used to weakly linearize the groupoid around any embedded saturated submanifolds  \cite{dhf1}. 





Following \cite{dhf2}, we say that two metrics $\eta_1$ and $\eta_2$ on $G\toto M$ are {\bf equivalent} if they induce the same inner product on the normal vector spaces $N_xO$ for any $x\in M$.
More generally, we define a {\bf Riemannian Morita map} $f:(\tilde G\toto\tilde M,\tilde\eta)\to(G\toto M,\eta)$ as a Morita map between Riemannian groupoids that induces isometries on the normal vector spaces $N_{\tilde x}\tilde O\to N_{f(\tilde x)}O$.
Then $\eta_1$ and $\eta_2$ are equivalent if and only if the identity is a Riemannian Morita map:
$$(G\toto M,\eta_1)\xto{\id}(G\toto M,\eta_2)$$


It is proven in \cite{dhf2} that if $f:(H\toto N)\to(G\toto M)$ is a Morita map and $\eta_G$ is a metric on $G\toto M$, then there exists one $\eta_H$ in $H\toto N$ making $f$ a Riemannian Morita map, that this $\eta_H$ is unique up to equivalence, and moreover, that this pullback construction sets a 1-1 correspondence between equivalence classes of metrics.
A {\bf stacky metric} $[\eta]$ on $[M/G]$ is then defined as the equivalence class of a metric on the groupoid $G\toto M$. 
This generalizes the usual notions of metrics for manifolds and orbifolds.



\begin{example}\label{ex:submersion3}
Following Example \ref{ex:submersion2}, if $G\toto M$ is the submersion groupoid arising from $q:M\to B$, then 
a metric $\eta$ on $G\toto M$ induces  $\eta^M$ on $M$ and $\eta^B$ on $B$ making $q$ a Riemannian submersion, and two metrics $\eta,\eta'$ are equivalent if and only if the induced metrics on $B$ agree. Moreover, any metric $\eta^B$ on $B$ can be induced from a groupoid metric on the submersion groupoid \cite{dhf2}.
Thus, stacky metrics on $[M/G]$ correspond to metrics on $B$.
\end{example}

\begin{example}\label{ex:orbifold3}
If $G\toto M$ is an {\'e}tale Lie groupoid, namely one on which $G, M$ have the same dimension, then a Riemannian structure on $G\toto M$ is the same as a metric on $M$ that is invariant under local bisections, and two metrics are equivalent if and only if they are equal. Thus, following Example \ref{ex:orbifold2}, if $G\toto M$ is proper and {\'e}tale, then $[M/G]$ is an orbifold, and stacky metrics on it agree with the orbifold metrics as classically defined \cite{gh,th}.
\end{example}

\begin{example}\label{ex:action3}
Given $K\action (M,\eta^M)$ an isometric action of a Lie group, a metric $\eta$ can be built on the action groupoid $G\toto M$ by an averaging process described in \cite{dhf2}.
This process involves many choices, and the resulting metric $\eta^{(0)}$  on $M$ that does not agree with $\eta^M$ in general. Still, they do agree on the normal directions, so the equivalence class of $\eta$ does not depend on the choices made. Our stacky metrics allow us to do Riemannian geometry on $M/K$ (see \cite{aalm}).
\end{example}

\begin{example}\label{ex:foliation3}
Let $F$ be a foliation on $M$. A groupoid metric $\eta$ on its holonomy groupoid $\hol(F)\toto M$ induces a metric on $M$ that makes $F$ a {\bf Riemannian foliation}, meaning that if a geodesic is normal to an orbit at a given time then it remains normal to the orbits at every time. Conversely, starting with $\eta^M$ a metric for which $F$ is Riemannian, it is possible to build a groupoid metric $\eta$ out of it, and two such metrics are going to be equivalent \cite{dhf1}. 
\end{example}

\section{Curves on Riemannian stacks}
\label{sec.curves}


One of the main advantages of working with differentiable stacks is that they provide a clean notion for smooth maps, and in particular, for smooth curves on quotient spaces. In this section we study stacky curves, show that their speeds vary continuously, and prove our first main theorem, asserting that for separated Riemannian stacks the normal length of curves recovers the canonical distance defined on the coarse orbit space. We also derive important corollaries.

\subsection{Stacky curves}
\label{subsec:stacky-curves}



Let $G\toto M$ be a Lie groupoid and $[M/G]$ its orbit differentiable stack. We define a {\bf stacky curve} $\alpha:I\to [M/G]$ as a stacky map from a real interval, viewed as a stack via the unit groupoid $I\toto I$. 
Such a curve $\alpha$ is the class of a cocycle $(a_{ji})$ supported over some open cover $\U=(U_i)$ of the interval $I$. 
We use the notations $a_{ji}:I\to G$ and $a_{ii}=a_i:U_i\to M$.
$$(I\toto I)\underset{\sim}{\xfrom{\pi_{\U}}} (\coprod U_{ji}\toto \coprod U_i) \xto{(a_{ji})} (G\toto M)$$
The curves $a_{ji}$ satisfy the cocycle condition $a_{kj}(x)a_{ji}(x)=a_{ki}(x)$ for $x$ on  triple intersections $U_{kji}$. 
We call $(a_{ji})$ a {\bf good cocycle} if it is supported on a dimension 1 cover, namely $i\in\Z$, $U_{ji}=\emptyset$ except for consecutive $i,j$ and there are no triple intersections $U_{kji}$ for different $i,j,k$. Of course, any curve can be presented by a good cocycle, and two good cocycles define the same curve if they restrict to isomorphic good cocycles on a common refinement.

\begin{example}\label{ex:submersion4}
If $G\toto M$ is the submersion groupoid arising from $p:M\to B=M/G$ as in Example \ref{ex:submersion2}, then a good cocycle $(a_{ji}):I_\U\to G$ is  a collection of local lifts $a_i$ to $M$ of a given curve on $B$. In this case the transitions $a_{i+1,i}$ are completely determined by the segments $a_{i}$, for there is no isotropy. Two cocycles are equivalent if the curve on $B$ is the same.
\end{example}

\begin{example}\label{ex:orbifold4}
Let $G\toto M$ be a proper {\'e}tale groupoid and $O=[M/G]$ its orbit orbifold (cf. Example \ref{ex:orbifold2}). A curve $\alpha:I\to O$ is classically defined as a continuous curve $\alpha:I\to |O|=M/G$ that can be locally lifted to smooth curves $a_i:I_i\to U_i$ on orbifold charts \cite[2.4]{mmbook}.
A stacky curve $\alpha:I\to[M/G]$ induces a curve in this classic sense, for if $(a_{ji})$ is a cocycle representing it, then the segments $a_{i}$ serve as local lifts into orbifold charts. But the classic notion of curve is too sloppy and has not a clear interpretation in terms of stacks. For instance, the smooth curves
$$\alpha_{\pm}:\R\to\R^2 \qquad 
\alpha_{\pm}(t)=\begin{cases}
                 (t,e^{1/t})& t\leq 0\\
                 (t,\pm e^{-1/t})& t\geq 0
                \end{cases}$$
define the same curve on the quotient $\R^2/\Z_2$ defined by the reflection along the $x$-axis, but since $\alpha_+$ and $\alpha_-$ are not locally related by the action, then they are different as stacky curves. 
%
%
\end{example}

\begin{example}\label{ex:action4}
Following Example \ref{ex:action2}, given $K\action M$ an action and $G\toto M$ the resulting action groupoid, we can rewrite a good cocycle $(a_{ji})$ for a stacky curve $\alpha:I\to [M/G]$ as a family of curves $a_{i}:U_i\to M$ and $g_{i+1.i}:U_{i+1,i}\to K$ such that $g_{i+1,i}(t)a_i(t)=a_{i+1}(t)$ for $t\in U_{i+1,i}$. The collection $(g_{ji})$ defines a $K$-cocycle over the covering $(U_i)$ of the interval, and since every principal $K$-bundle over $I$ is trivial, we can integrate the cocycle, gaining a global representative $a:I\to M$ for any stacky curve $\alpha$. In other words, we are just translating the the local pieces $a_i$ using the action of $K$ to a get a single $a$ (compare with  \cite[Thm 6.6]{cdhp}).
\end{example}

\begin{example}\label{ex:foliation4}
A regular foliation $F$ on a manifold $M$ can be described by a family of submersions $f_i:V_i\to W_i\subset\R^q$, where $V_i\subset M$, such that when $V_{ji}\neq\emptyset$ there exists a diffeomorphism 
$\gamma_{ji}:W_i\to W_j$ 
satisfying $f_j=\gamma_{ji}f_i$ \cite{bh}.
It is enough to consider an open cover $V_i$ such that if $V_{ji}\neq\emptyset$ then $V_i\cup V_j$ is included in a foliated chart.
If we use $\hol(F)\toto M$ to make sense of $[M/F]$ as a stack 
(cf. Example \ref{ex:foliation2}), and we fix the defining submersions $f_i:V_i\to W_i$ as before, then a stacky curve $\alpha:I\to [M/F]$ can be reinterpreted as a cocycle $(a_{ji})$ with $(U_i)$ a good open cover of $I$, $a_i:U_i\to V_i$ and $a_{ji}(t)$ the germ of $\gamma_{ji}$ at $a_i(t)$. The relevant information of each segment $a_i$ is in the composition $b_i=f_ia_i:U_i\to W_i$, and a stacky curve on the leaf space is the same as a family of curves $b_i:U_i\to W_i$ that are connected by the defining cocycle $(\gamma_{ji})$.
\end{example}



Given $G\toto M$ a Lie groupoid, the {\bf velocity} of a stacky curve $\alpha: I\to [M/G]$ at $t_0\in I$ is 
$\alpha'(t_0)=d_{[t_0]}\alpha[\partial_t\r{t_0}]=[a'_{k}(t_0)]\in T_{[a_k(t_0)]}[M/G]$, 
where $(a_{ji})$ is a cocycle for $\alpha$ and $t_0\in U_k$.
Given a groupoid metric $\eta$ on $G\toto M$, it yields a $G_x$-equivariant inner product on $N_xO_x=T_xM/T_xO$, so we can define the normal norm $\n{v}_N$ of any vector $v\in T_xM$, and set the {\bf speed} of $\alpha$ at $t_0$ as $\n{\alpha'(t_0)}=\n{a'(t_0)}_N$. 
This is well-defined, it does not depend on the cocycle $(a_{ji})$ representing $\alpha$, for the groupoid version of the normal representations are by isometries.

\begin{remark}
A good cocycle $(a_{ji})$ for a stacky curve should be compared with the notion of $G$-path \cite[2.3]{gh}, \cite[3.3]{mm2}. They are defined as a sequence alternating continuous paths $x_k\overset{\gamma_k}\leadsto y_k$ on $M$ and arrows $y_k\xto{g_k} x_{k+1}$ in $G$. Given a good cocycle $(a_{ji})$ for a stacky curve we can build a $G$-path by splitting the interval, choosing $t_k\in U_{k+1,k}$, and setting $\gamma_k=a_{k}\r{[t_k,t_{k+1}]}$ and $g_k=a_{k+1,k}(t_k)$. Conversely, a $G$-path on which every $\gamma_k$ is smooth gives rise to a good cocycle by first extending $g_{k}$ to a smooth curve $\tilde g_k:(t_k-\epsilon,t_k+\epsilon)\to G$, $\tilde g_k(t_k)=g_k$, and then modifying $\gamma_k$ and $\gamma_{k+1}$ near $t_k$ so as to agree with $s\tilde g_k$ and $t\tilde g_k$. Even though these operations depend on choices, they are well-defined up to equivalence classes of cocycles and small {\em deformations of $G$-paths} \cite{mm2}.
The advantages of our cocycles is that they fit the general theory of stacky maps, without the need of an ad-hoc definition, and they allow us  to make sense of smoothness.
\end{remark}



\subsection{Continuity of the normal speed}


We will define the length of a stacky curve in a Riemannian stack as the integral of its speed. For this to make sense, we need to show that the speed varies continuously, and this is rather subtle when the dimensions of the normal directions vary. This subsection takes care of this technical issue. 


Let $(G\toto M,\eta)$ be a Riemannian groupoid, $[M/G]$ its orbit stack, and $\alpha:I\to [M/G]$ a stacky curve, presented by a good cocycle $(a_{ji})$ supported over some open cover $(U_i)$. We want to show that the speed $\n{\alpha'(t)}$, as defined in previous subsection, varies continuously on $t$. Working locally, we can assume that $\alpha$ is in fact presented by a curve $a:I\to M$, and as discussed before, $\n{\alpha'(t)}=\n{a'(t)}_N$. We will first show the continuity at a fixed point of the groupoid, and then extend to the general case by passing to a transversal.


\begin{proposition}\label{prop:cont-speed}
Let $(G\toto M, \eta)$ be a Riemannian groupoid and $a:I\to M$ a smooth curve. Then the normal component of the speed $\n{a'(t)}_N$ is continuous at any $t_0$ in $I$.
\end{proposition}

\begin{proof}
{\it Case the orbit of $a(t_0)$ is $0$-dimensional.}
If $a'(t_0)=0$ then the result follows from the inequality $0\leq \n{a'(t)}_N\leq \n{a'(t)}$.
Suppose then that $a'(t_0)\neq0$, assume that $t_0=0$, and working in normal coordinates given by the metric $\eta$ around $a(t_0)\in M$, we can further assume that $a(t_0)=0$, that $M=B\subset\R^n$ is a small ball centered at $0$, that the matrix $A(t)$ of the metric $\eta$ is the identity at $0$ and has vanishing first derivatives at $0$, and that the geodesic spheres around $t_0$ agree with the usual spheres.

The radial vector field $R$ is orthogonal to the characteristic foliation $F$, for its integral curves are geodesics, and $F$ is singular Riemannian.
Let $\theta(t)\geq0$ be the angle between $a'(t)$ and its projection to $F^\bot$. Then $\theta(0)=0$, and since $\n{a'(t)}_N=\cos(\theta(t))\n{a'(t)}$, to show that $\n{a'(t)}_N$ is continuous at $0$ it is enough to show that $\lim_{t\to 0}\theta(t)=0$, or in other  words, that $a'(t)$ is almost orthogonal to $F$ near $0$. We do this by computing the angle $\tilde\theta(t)$ between $R$ and $a'(t)$, for $\theta\leq\tilde\theta$, as it follows from the fact that $R_x$ belongs to $F_x^\bot$ for all $x\in B$.

Writing $a'(0)=v$ we have $a(t)=tv+O(t^2)$ and $a'(t)=v+ O(t)$. From
$$\eta\left({a'(t),R_{a(t)}}\right)=
[v+O(t)]^t[I+ O(t^2)][tv+O(t^2)]=
t\n{v}^2+ O(t^2)$$
and
$$\n{a'(t)}^2_{\eta}\n{R_{a(t)}}^2_{\eta}=
\p{v+O(t),v+O(t)}\p{tv+O(t^2),tv+O(t^2)}=
t^2\n{v}^4+O(t^3)$$
we conclude 
$\lim_{t\to 0} \cos^2(\tilde\theta(t))=
\lim_{t\to 0}\frac{t^2\n{v}^4+O(t^3)}{t^2\n{v}^4+O(t^3)} = 1$
and the result follows.


{\it Case the orbit of $a(t_0)$ is not $0$-dimensional.}
The image of a small ball $B\subset N_xO$ centered at $0$ under the exponential map is a submanifold $S$ transverse to each orbit it crosses, and $T_xS\cong N_xO$.
It easily follows from \cite[Prop 3.5.1]{dh} that the map $s:t^{-1}(S)\to U\subset M$ is submersive onto an invariant open, and that the restriction groupoid $G_S\toto S$ is a Lie subgroupoid of $G$.
Let $h$ be a smooth local section for $s:t^{-1}(S)\to U\subset M$ such that $h(x) = u(x)$. Then $b(t) = th(a(t))$ is a curve in $S$ with $b(t_0)=x$ which
is naturally isomorphic to $a$, and therefore, such that $\n{b'(t)}_N = \n{a'(t)}_N$ for all $t$ close to $t_0$. 
We can now use the 0-dimensional case with the curve $b$ in the groupoid $G_S\toto S$, regarded with a groupoid metric that makes the inclusion $(G_S\toto S)\to (G_U\toto U)$ a Riemannian Morita map (cf. \cite[Thm 6.3.3]{dhf2}). This groupoid metric may not agree with the one $G_S$ inherits as a submanifold, at least a priori, but it does insure that the norms of normal vectors is preserved, so $\n{b'(t)}_N$ is the same in $G_S$ as in $G_U$. 
\end{proof}



As an immediate corollary of previous proposition, we establish the following technical fundamental result, necessary to make sense of lengths of curves in a stack.

\begin{corollary}\label{cor:cont-speed-stack}
The speed of a stacky curve $\alpha:I\to[M/G]$ on a Riemannian stack varies continuously on $t$.
\end{corollary}

\begin{proof}
If $\alpha$ is given by a cocycle $(a_{ji})$ supported over $(U_i)$ and $t_0\in U_k$, then in a neighborhood of $t_0$ we have $\n{\alpha'(t)}=\n{a_k(t)}_N$ and this is continuous by Proposition \ref{prop:cont-speed}. 
\end{proof}


We remark that even $\theta(t)$ is continuous at a fixed point, it may fail to be differentiable, as we will show in next example. This example also shows that, given $(G\toto M,\eta)$ a Riemannian groupoid and given $a:I\to M$, there may not exist a curve $b:I\to M$ isomorphic to $a$ that is orthogonal to the characteristic foliations, otherwise $\n{a(t)}_N=\n{b(t)}$ would indeed be smooth.

\begin{example}
Let $M=\R^2$ with the $S^1$-action by rotations. The resulting action groupoid $G\toto M$ can be endowed with a groupoid metric $\eta$ such that $\eta^{(0)}$ is the standard metric on $\R^2$ \cite[4.9]{dhf1}. Let $a:I\to M$ be the parabola curve $a(t)=(t,t^2)$ on $\R^2$. Then $a(0)=0$ is a fixed point,  
$\theta(t)=\cos^{-1}(\frac{t+2t^3}{\sqrt{1+4t^2}\sqrt{t^2+t^4}})$ for $t\neq0$ and $\theta(0)=0$ is continuous at $0$ but not differentiable.
\end{example}

\subsection{Normal distance by stacky curves}


In a Riemannian manifold $(M,\eta^M)$ the length $\ell(a)$ of a curve $a:I\to M$ is the integral of its speed. If $M$ is connected, then a distance $d(x,y)$ can be defined, as the infimum of the length of curves connecting $x$ and $y$. We call it the {\bf metric distance}.
$$\ell(a)=\int_I \n{a'(t)} dt \qquad 
d(y,x)=\inf\{\ell(a): a \text{ smooth, } x,y\in a(I)\}$$




Given a Riemannian groupoid $(G\toto M,\eta)$ with coarse orbit space $M/G$  connected, a {\bf normal pseudo-distance} $d_N$ on $M/G$ can be defined, by considering {\bf chains} $(x_0,\dots,x_{2n+1})$ such that $x_{2i},x_{2i+1}$ are in the same orbit and $x_{2i-1},x_{2i}$ are in the same component, and setting
$$d_N([x],[y])=
\inf\bigg\{\sum_{i=1}^{n} d(x_{2i-1},x_{2i}): (x_i) \text{ chain from $x$ to $y$}\bigg\}$$

This pseudo-distance $d_N$ was studied by \cite{ppt}, for a proper Lie groupoid $G\toto M$ equipped with a metric on the units that is transversally invariant. 
They establish some properties of $d_N$ (cf. \cite[Thm 6.1]{ppt}), such that $d_N([x],[y])=d(O_x,y)$ for $y$ in a small tubular neighborhood of $O_x$, from where it  follows that $d_N$ is a distance. Moreover, since $d_N([x],[y])\leq d(x,y)$, they remark that the orbit space $M/G$ inherits some properties from $M$, like being a length space, and an  Alexandrov space when $(M,\eta)$ has bounded curvature.


We have seen in Proposition \ref{prop:cont-speed} that the speed of a stacky curve varies continuously. This is enough to make sense of the {\bf length} of a stacky curve $\alpha: I\to [M/G]$, defined as the integral of the speed,
$\ell(\alpha)=
\int_{I} \n{\alpha'(t)} dt$.
If $\alpha$ is presented by a good cocycle $(a_{ji},U_i)$, then we have
$$\ell(\alpha)=
\int_{I} \n{\alpha'(t)} dt=
\sum_{i}\int_{U_i}\n{a'_{i}(t)}_Ndt - 
\sum_{i}\int_{U_{i+1,i}}\n{a'_{i+1,i}(t)}_Ndt$$
This is because there are no triple intersections, and the first sum is counting twice each overlap $U_{i+1,i}$. 
Note that for $t\in U_{i+1,i}$ the source and target yield isometries
$N_{a_{i+1,i}(t)}G_O\to N_{a_{i}(t)}O$ and
$N_{a_{i+1,i}(t)}G_O\to N_{a_{i+1}(t)}O$
mapping $[a'_{i+1,i}(t)]$ to $[a'_{i}(t)]\in $ and $[a'_{i+1}(t)]\in N_{a_{j}(t)}O$, respectively. 


Our first main theorem shows that the quotient pseudo-distance $d_N$ can be measured with stacky curves. This offers us an alternative approach to the results in \cite{ppt}, and the opportunity to strengthen some of them.

\begin{theorem}\label{thm:distance-length}
If $(G\toto M,\eta)$ is a Riemannian groupoid and $M/G$ is connected, 
then $d_N([x],[y])$ is the infimum of the lengths $\ell(\alpha)$ of stacky curves $\alpha:I\to[M/G]$ connecting $[x]$ and $[y]$.
\end{theorem}

\begin{proof}
Fix $x,y$ points in $M$.
Given $\epsilon>0$, we will build a stacky curve $\alpha:I\to [M/G]$ connecting $[x],[y]$ and such that $\ell(\alpha)< d_N([x],[y])+\epsilon$.
By definition of $d_N$, we know of the existence of a chain $(x_0,\dots,x_{2n+1})$
such that $x_0=x$, $x_{2n+1}=y$ and
$\sum_{i=1}^{n} d(x_{2i-1},x_{2i})<d_N([x],[y])+\epsilon/2$.
The idea is to pick for each $i$ a curve $a_i:U_i\to M$ connecting $x_{2i-1}$ and $x_{2i}$ and such that 
$\ell(a_i)\leq d(x_{2i-1},x_{2i})+\frac{\epsilon}{2n}$, and then combine the several $a_i$ into a cocycle $(a_i)$ for a stacky curve $\alpha$. 
We can do this inductively. After picking $a_i$, let $a_i(t_{2i})=x_{2i}$, let 
\smash{$x_{2i+1}\xfrom{g_i}x_{2i}$} be an arrow in $G$, and let $a_{i+1,i}:J_i\to G$ be a local lift of $a_i$ along $s$ such that $a_{i+1,i}(t_{2i})=g_i$. Then $ta_{i+1,i}(t_{2i})=x_{2i+1}$, and if $J_i$ is small enough,  we can pick the next curve $a_{i+1}$ so as to agree with $ta_{i+1,i}$ in $J_i=U_{i+1,i}$ and still satisfy $\ell(a_i)\leq d(x_{2i-1},x_{2i})+\frac{\epsilon}{2n}$. This way $\alpha=(a_{ji})$ satisfy 
$$\ell(\alpha)\leq \sum_{i=1}^{n}\ell(a_i)\leq \sum_{i=1}^{n}\big(d(x_{2i-1},x_{2i})+ \epsilon/2n\big)\leq d_N([x],[y])+\epsilon.$$

%


So far, we have shown that $\inf\{\ell(\alpha): [x],[y]\in\alpha(I)\}\leq d_N([x],[y])$. For the converse, given $\alpha:I\to[M/G]$ a stacky curve connecting $[x]$ and $[y]$, let us show that $(1+\epsilon)\ell(\alpha)$ must be greater or equal than $d_N([x],[y])$, for arbitrary $\epsilon>0$. Because of the additivity of $\ell$ and the triangle inequality of $d_N$, we can subdivide $I$ in small intervals and show the inequality for each piece of curve. 
This allow us to work locally,
and without loss of generality, we can assume that $\alpha$ is presented by a single smooth curve $a:I\to M$.
Given $t_0\in I$, we can assume that $\n{a'(t_0)}=\n{a'(t_0)}_N$, otherwise we can substitute $a$ by an isomorphic curve satisfying this, as shown along the proof of Proposition \ref{prop:cont-speed}. 
Since $\n{a'(t)_N}$ is continuous, we have that
$
\n{a'(t)}\leq (1+\epsilon) \n{a'(t)}_N
$
in a neighborhood of $t_0$. By compacity we can cover the interval $I$ with finitely many such neighborhoods, and get
$$d_N([x],[y])\leq \sum_i d_N([a(t_i)],[a(t_{i+1})])\leq 
\sum_i \int_ {t_i}^{t_{i+1}} (1+\epsilon)\n{a'(t)}_Ndt =(1+\epsilon)\ell(\alpha).$$
Since $\epsilon$ is arbitrary, we have shown $d_{N}([x],[y])\leq \ell (\alpha)$ and the proof is completed.
\end{proof}


We close with some immediate corollaries of our characterization of $d_N$ by using stacky curves.

\begin{corollary}\label{cor:Morita-distance}
If $\phi:(\tilde G\toto \tilde M,\tilde\eta)\to (G\toto M,\eta)$ is a Riemannian Morita map, then the map $\bar\phi:\widetilde{M}/\widetilde{G}\to M/G$ between the coarse orbit spaces preserves the pseudo-distance.
\end{corollary}

\begin{proof}
The map $M/G\to \tilde M/\tilde G$ is a bijection, and moreover, $\phi$ yields a 1-1 correspondence between stacky curves $\alpha$ in $[\tilde G/\tilde M]$ and $\phi\alpha$ in $[M/G]$. Besides, since $\phi$ is also Riemannian, this correspondence preserves the length of stacky curves, and the corollary follows by Theorem \ref{thm:distance-length}.
\end{proof}

\begin{corollary}
Equivalent metrics on the same Lie groupoid $G\toto M$ yield the same pseudo-distance on the coarse orbit space $M/G$.
\end{corollary}


\begin{corollary}
The pseudo-distance $d_{N}$ on $M/G$ is a Riemannian Morita invariant, it depends only on the Riemannian stack $([M/G],[\eta])$.
\end{corollary}


\begin{remark}
If $d_N$ is a distance, then $(M/G,d_N)$ clearly inherits a quotient length structure from $(M, \eta^M)$ (cf. \cite{bbi}, pag 63). The Theorem \ref{thm:distance-length} says that the quotient length structure can be recovered as the length structure given by the (traces of) stacky curves on $[M/G]$. 
\end{remark}


\section{Geodesics on Riemannian stacks}
\label{sec.geodesics}


This is the main section of the paper. We introduce here our notion for stacky geodesics, discuss fundamental examples, and set the existence and uniqueness of geodesics with prescribed initial conditions. We also prove our second main theorem, showing that, on a separated Riemannian stack, geodesics can be characterized as minimizing curves. We finish with our third main theorem, a result on complete Riemannian stacks, a stacky version of Hopf-Rinow. 

\subsection{Definitions and examples}


Recall that a singular foliation $F$ on a Riemannian manifold $(M,\eta)$ is Riemannian if every geodesic that is normal to $F$ at a given time remains normal to $F$ at every time \cite{abt}. We call such a geodesic a {\bf normal geodesic}. In a Riemannian groupoid $(G\toto M,\eta)$ the foliation $F_M$ on $M$ by the orbits and its pullback $F_G=s^*F_M=t^*F_M$ on $G$ are singular Riemannian \cite{dhf1}. 
We say that a curve $a:I\to G$ is a normal geodesic if it is so with respect to $F_G$. If $a(I)\subset M$, this is equivalent to be a normal geodesic for $F_M$ in $M$, for $M\subset G$ is totally geodesic.


Given a Riemannian groupoid $(G\toto M,\eta)$ as before, we consider $[M/G]$ its associated stack, and we define a stacky curve $\alpha: I\to [M/G]$ to be a {\bf stacky geodesic} if it can be represented by a cocycle $\alpha=(a_{ji})$ on which each $a_{ji}:I\to G$ is a normal geodesic. 
$$(I\toto I)\underset{\sim}{\xfrom{\pi}} (\coprod U_{ji}\toto \coprod U_i) \xto{(a_{ji})} (G\toto M)$$
Note that the curves $a_{i}:U_i\to M$ are also normal geodesics of $M$. Every geodesic can be represented by a good cocycle of normal geodesics, and two geodesics are equivalent if they yield the same good cocycle on a common refinement.


\begin{remark}\label{rmk:good-definition}
We will prove as a corollary of Theorem \ref{thm:geodesics-distance} that for separated Riemannian stacks, our main case of interest, the notion of stacky geodesic only depends on the class $[\eta]$, hence being a {\em well-defined} notion.   
For arbitrary Riemannian stacks, we conjecture this is still true.
In any case, our definition still makes sense, as one depending on the groupoid metric $\eta$. 
\end{remark}



\begin{example}
Let $p:M\to B$ be a submersion, $M\times_BM\toto M$ the corresponding submersion groupoid, and $\eta$ a groupoid metric on it, which yields a metric on $B$ making $p$ a Riemannian submersion (cf. Example \ref{ex:submersion3}). 
If $\alpha=(a_{ji}):I\to [M/G]$ is a stacky geodesic, then each $a_i:I\to M$ is a horizontal geodesic and projects to a geodesic on $B$. Conversely, given a geodesic on $B$, the horizontal local lifts are geodesics, and can be used to build a cocycle for a stacky geodesic $(a_{ji}):I\to[M/G]$ 
(cf. Example \ref{ex:submersion4}). 
Thus, our stacky geodesics extend the usual notion for manifolds.
\end{example}

\begin{example} 
Let $G\toto M$ be a proper {\'e}tale Riemannian groupoid and $O=[M/G]$ its orbit Riemannian orbifold (cf. Example \ref{ex:orbifold2}). An orbifold geodesic $\alpha:I\to O$ is classically defined as a continuous curve $\alpha:I\to |O| =M/G$ that locally lifts to a smooth geodesic $a_i:J\to U_i$ into some Riemannian orbifold chart $(U_i, G_i, \phi)$.
This ad hoc definition for orbifolds matches our general approach, for the local lifts $(a_i)$ provide a cocycle for a stacky geodesic $I\to [M/G]$ at the level of objects, and it can always be extended at the level of arrows (see Corollary \ref{cor:merge-geodesics}).
\end{example}

\begin{example}\label{ex:action5}
For $K \action (M,\eta^M)$ an isometric action, we have seen how to extend $\eta^M=\eta^{(0)}$ to a groupoid metric $\eta$ on the action groupoid $G\toto M$ using a right invariant metric $\eta^K$ of $K$ (cf. Example \ref{ex:action3}).
By construction, $\eta^{(1)}$ on $G=K\times M$ is such that the following are Riemannian submersions.
$$\xymatrix{
(K\times K\times M,\eta^K\times\eta^K\times\eta^M) \ar[rr]^{\pi_2} 
\ar[d]^{(k_2,k_1,x)\mapsto (k_2k_1^{-1},k_1x)}
& &  (K\times M,\eta^K\times\eta^M) \ar[d]^{\pi}\\
(G,\eta^{(1)})\ar[rr]^s& &(M,\eta^M)}$$
By diagram chasing normal geodesics on previous square, we can see that a normal geodesic $a:I\to G$ must be of the form $a=(k,b)$ with $k\in K$ constant and $b$ a normal geodesic in $M$. 
Thus, in a stacky geodesic $\alpha=(a_{ji}):I\to [M/G]$, each $a_i$ gives a normal geodesic in $M$, and each $a_{ji}$ gives a $k_{ji}\in K$ such that $k_{ji}.a_i\r{U_{ji}}=a_j\r{U_{ji}}$. Since the $K$ action is by isometries, we conclude that any such $\alpha$ can be presented by a single normal geodesic $a:I\to M$ (cf. Example \ref{ex:action4}, compare with \cite{aalm}).
\end{example}

\begin{example}
Analogous to Example \ref{ex:foliation4}, a Riemannian foliation $(M,F,\eta^M_0)$ can be described by Riemannian submersions $f_i:(V_i,\eta^{V_i})\to(W_i,\eta^{W_i})$ with $V_i\subset M$ and isometries 
$\gamma_{ji}:W_i\to W_j$ 
satisfying $f_j=\gamma_{ji}f_i$ whenever $V_{ji}\neq\emptyset$.
Using the holonomy groupoid $(\hol(F)\toto M,\eta)$ to make sense of $[M/F]$ (cf. Example \ref{ex:foliation3}), and fixing the defining Riemannian submersions $f_i:(V_i,\eta^{V_i})\to (W_i,\eta^{W_i})$, we can present a stacky geodesic $\alpha: I \to [M/F]$ by normal geodesics $a_i:U_i\to V_i$ satisfying 
$f_{j}a_{j}=\gamma_{j,i}f_ia_i$. 
The transversal data of each segment $a_i$ is captured by the geodesics $b_i:U_i\to W_i$. It follows that a geodesic on the leaf space is the same as a family of geodesics $b_i:U_i\to(W_i,\eta^{W_i})$ that are glued by the defining cocycle $\gamma_{i+1,i}$ (compare with \cite{aj}).
\end{example}


Geodesic on manifolds are defined infinitesimally in terms of the Levi-Civita connection. Since we have not developed stacky connections, we are not in conditions to prove such an infinitesimal description. Still, we can prove that our notion is local, and this will be useful later.

\begin{lemma}\label{lemma:local-geodesic}
A stacky curve $\alpha:I\to [M/G]$ is a stacky geodesic if and only if for every $t\in I$ there is a smaller interval $t\in J\subset I$ such that $\alpha|_J$ is a stacky geodesic.
\end{lemma}

\begin{proof}
A restriction of a geodesic is clearly a geodesic, for a cocycle of normal geodesics restricts to another of the same type. 
For the converse, given a stacky curve $\alpha:I\to [M/G]$ that is locally a geodesic, and therefore it can be locally presented by good cocycles of normal geodesics, we want to merge them into a good cocycle of normal geodesics for $\alpha$.
The key step is, 
given $a:I\to M$ and $b:J\to M$ normal geodesics and given $t\in I\cap J$ such that $a(t)$ and $b(t)$ are in the same orbit $O$ and $[a'(t)]=[b'(t)]$ in $NO/NG_O$, to find a normal geodesic $c:(t-\epsilon,t+\epsilon)\to G$ such that $sc=a$ and $tc=b$. From $[a'(t)]=[b'(t)]$ we get that there is some arrow $b(t)\xfrom g a(t)$ such that $g[a'(t)]=[b'(t)]$ \cite[3.5.1]{dh}. Then, if $v\in T_gG$ is the horizontal lift of $a'(t)$ to $g$ via the source, the geodesic with initial condition $v$ will serve as $c$.
%
%
\end{proof}


\begin{corollary}\label{cor:merge-geodesics}
If $\alpha:I\to [M/G]$ and $\beta:J\to [M/G]$ are stacky geodesics on a separated Riemannian stack and there is $t_0\in I\cap J$ such that $\alpha(t_0)=\beta(t_0)$ and $\alpha'(t_0)=\beta'(t_0)$, then they can be glued into a geodesic $\xi: I\cup J\to [M/G]$ that extends both.
\end{corollary}

\begin{proof}
It follows from the key step described in previous lemma.
%
\end{proof}

\subsection{Existence and uniqueness}


Let $(G\toto M,\eta)$ be a Riemannian Lie groupoid. Given $[x]$ in $M/G$ and $[v]\in T_{[x]}[M/G]$, in this section we deal with the questions of whether there exists a stacky geodesic $\alpha:I\to[M/G]$ passing through $[x]$ with velocity $[v]$, and in case it exists, whether it is unique. 

\begin{lemma}
For every $[x]\in M/G$ and $[v]\in T_{[x]}[M/G]$ there exists a stacky geodesic $\alpha:(-\epsilon,\epsilon)\to[M/G]$ such that $\alpha(0)=[x]$ and $\alpha'(0)=[v]$. If $\beta:(-\epsilon,\epsilon)\to[M/G]$ is another stacky geodesic with $\beta(0)=[x]$ and $\beta'(0)=[v]$, then $\alpha|_{(-\epsilon',\epsilon')}=\beta|_{(-\epsilon',\epsilon')}$ for some $\epsilon'>0$. 
\end{lemma}

\begin{proof}
Pick $x\in M$ representing $[x]$ and $v\in T_xM$ that is normal to the orbit $O_x$ and represents $[v]$. 
Then the geodesic $a:(\epsilon,\epsilon)\to M$ on $M$ with $a(0)=x$ and $a'(0)=v$ is normal and represents a stacky geodesic $\alpha':(-\epsilon,\epsilon)\to[M/G]$ with the requested initial conditions. 

Regarding uniqueness, and working locally, we can represent $\alpha,\beta$ by normal geodesics $a,b:I\to M$. Since $[a'(0)]=[b'(0)]\in T_{[x]}[M/G]$, we can pick an arrow \smash{$b(0)\xfrom{g}a(0)$} such that $g\cdot [a'(0)]=[b'(0)]$,
consider $v\in T_gG$ the normal lift of $a'(0)$ along the source map, 
and take $c: (-\epsilon',\epsilon')\to G$ the geodesic of $G$ with $c(0)=g$ and $c'(0)=v$. Then $c$ yields an isomorphism between the restrictions of $a$ and $b$. 
\end{proof}


While local existence and uniqueness follow straightforward from the properties of geodesics on manifolds, the global existence is more subtle, and does not hold for every Riemannian groupoid, as we can see in next example.

\begin{example}
Let $F$ be the foliation in $M=\R^2\setminus\{0\}$ given by the vertical lines, and consider its holonomy groupoid $\hol(F)\toto \R^2\setminus\{0\}$. The standard metric on $M$ can be extended to a groupoid metric (cf. Example \ref{ex:foliation3}). Then $[M/F]$ identifies with the non-Hausdorff real line with two origins, and the normal geodesics $a(t)=(t-1,1)$ and $b(t)=(t-1,-1)$ in $M=\R^2\setminus \{0\}$ are not isomorphic, for they do not define the same map on the orbit spaces. Then they yield different stacky geodesics with the same initial conditions. 
\end{example}


We will see that global uniqueness of geodesics holds when working with separated Riemannian stacks. Next proposition will help us to build isomorphisms between curves on a proper groupoid.

\begin{proposition}\label{prop:glob.uniq.geodesics}
Suppose $G\toto M$ is proper. If $a,b: I \to G$ are normal geodesics and $b(0)\xfrom{g}a(0)$ is an arrow with $g\cdot [a'(0)]=[b'(0)]$, then the normal geodesic $c$ on $G$ lifting $a$ to $g$ via the source map is defined over the whole $I$, and it gives an isomorphism between $a$ and $b$.
\end{proposition}

\begin{proof}
By reparametrizating if necessary we can assume that $a$ and $b$, and therefore $c$, are geodesics with speed 1. 
Let $(t_{min},t_{max})$ be the maximal domain of the geodesic $c$.
Note that $c$ restricted to $I\cap(t_{min},t_{max})$ is both an $s$-lift of $a$ and a $t$-lift of $b$, for $s$ and $t$ are Riemannian submersions, and $c$ is orthogonal to the fibers (see key step of Lemma \ref{lemma:local-geodesic}).
We want to show that $I\subset (t_{min},t_{max})$.
Suppose otherwise that $t_{max}\in I$, the other case $t_{min}\in I$ is analogous.
Consider a sequence $t_{n}\nearrow t_{max}$ so
$\{(c(t_n),c'(t_n))\}$ is contained in the compact
$$K=\{v\in T_gG: s(g)\in a[t_{max}-\epsilon,t_{max}], t(g)\in b[t_{max}-\epsilon,t_{max}], \n{v}=1\}.$$
Then there is a convergent subsequence $(c(t_{n_k}),c'(t_{n_k}))\to (g_0,v_0)$. 
By the local existence of the geodesic flow on the manifold $G$, there is a neighborhood $W$ of $(g_0,v_0)$ in $TG$ such that every geodesic with initial condition in $W$ is defined over the same interval $(-\delta,\delta)$. For $k$ sufficiently large $(c(t_{n_k}),c'(t_{n_k}))\in W$ and $t_{max}-t_{n_k}<\delta$, thus we can extend $c$ to the interval $(t_{min},t_{n_k}+\delta)$, contradicting the maximality of $t_{max}$.
\end{proof}


We can derive now the global uniqueness of geodesics on separated Riemannian stacks.

\begin{corollary}\label{cor:global-uniqueness}
If $\alpha,\beta:I\to [M/G]$ are stacky geodesics on a separated Riemannian stack with the same initial data and defined over the same interval then $\alpha=\beta$.
\end{corollary}

\begin{proof}
By refining coverings if necessary, we can assume that $\alpha$ and $\beta$ are given by good cocycles $(a_{ji})$, $(b_{ji})$ defined over the same good covering $\{ U_i\}$, and that $t_0\in U_0$. 
Since $\alpha'(t_0)=\beta'(t_0)$, we can pick an arrow 
$b_0(t_0)\xfrom{g_0}a_0(t_0)$ such that $g\cdot a'_0(t_0)=b'_0(t_0)$, and by Proposition \ref{prop:glob.uniq.geodesics}, we can promote $g_0$ to a natural isomorphism $c_0: U_0\to G$ between $a_0$ and $b_0$.
Choosing a time $t_{1}$ in $U_{10}$, we can define 
$g_{1}=b_{10}(t_{1})c_0(t_{1})a_{10}(t_{1})^{-1}$, so as to make the following commutative:
$$\xymatrix@R=10pt@C=10pt{
& a_0(t_{1}) \ar[dl]_{c_0(t_{1})} \ar[dr]^{a_{10}(t_{1})} & \\
b_0(t_{10}) \ar[dr]_{b_{10}(t_{1})} & & a_1(t_{1}) \ar[dl]^{g_{1}}\\
& b_1(t_{1}) & }$$
Note that $g_{1}\cdot a'_1(t_{1})=b'_1(t_{1})$, so we can apply again Proposition \ref{prop:glob.uniq.geodesics} to promote $g_{1}$ to a natural isomorphism $c_1:U_1\to G$ between $a_1$ and $b_1$. 
Moreover, $b_{10}(t)c_{0}(t)=c_1(t)a_{10}(t)$ for all $t\in U_{10}$, for the product and inverse of normal geodesics is again a normal geodesic, and both sides have the same initial conditions at $t_{1}$.
By iterating this argument, we get a sequence of arrows $g_{i}$, promote them to natural isomorphisms $c_i:a_i\cong b_i$, for any $i\geq 0$, and analogously for negatives $i$. The collections of $c_i$ are a natural isomorphism, showing that the cocycles $(a_{ji})$ and $(b_{ji})$ are equivalent.
\end{proof}

Finally, combining Corollaries \ref{cor:merge-geodesics} and \ref{cor:global-uniqueness}, we get:

\begin{corollary}
Given $[M/G]$ a separated Riemannian stack , $[x]\in M/G$ and $[v]\in T_{[x]}[M/G]$, there is a unique maximal stacky geodesic $\alpha: I\to [M/G]$ such that $\alpha(0)=[x]$ and $\alpha'(0)=[v]$. 
\end{corollary}


%


\subsection{Geodesics are locally minimizing}


Geodesics on manifolds can be characterized as curves that locally minimize the distance \cite[6.6-6.12]{lee}. 
A way to establish this relation between geodesics and distances, avoiding the calculus of variations, is by using Gauss lemma, which claims that the exponential map behaves as an isometry radially. Next we develop a stacky version of these results.



\begin{proposition}[Stacky Gauss Lemma]\label{prop:gauss-lemma}
Given $(G\toto M,\eta)$ a proper Riemannian groupoid and $x\in M$, 
there exists $\epsilon>0$ such that
$d_N([x],[\exp(v)])=\n{v}$ for all $v\in N_xO$ with $\n{v}<\epsilon$.
\end{proposition}

\begin{proof}
Let $B^{N}_{\epsilon}(x)\subset N_xO$ be the ball of normal vectors at $x$ of radius $\epsilon$. If $\epsilon$ is small then $S=\exp( B^{N}_{\epsilon}(x))$ is a slice at $x$: it is transversal to every orbit it meets, and $S\cap O_x$ is 0-dimensional. Since $G$ is proper $O_x$ is embedded, so we can further assume that $S\cap O_x=\{x\}$. Let $U$ be the open obtained by saturation of $S$. The groupoids $G_S\toto S$ and $G_U\toto U$ inherit Riemannian structures by restrictions, and therefore their own normal distances. Since $G_S\subset G_U$ is a Riemannian Morita map the map $(S/G_S,d_N)\to(U/G_U,d_N)$ preserves distances (see Corollary \ref{cor:Morita-distance}). We have thus reduced the problem to the fixed point case.

After eventually reducing $\epsilon$, we can assume that the exponential map gives a groupoid isomorphism, namely a groupoid linearization as below\cite{dhf1}:
$$\exp:(G_x\ltimes B^N_\epsilon(x)\toto B^N_\epsilon(x))\to(G_S\toto S)$$
We claim that the orbits $O\subset S$ are contained in the geodesics spheres around $x$. In fact, 
if $y,y'\in O$ then there are $v,v'\in B^N_\epsilon(x)$ and $g\in G_x$ such that $y=\exp(v)$, $y'=\exp(v')$ and $g.v=v'$. Since $G_x\action N_xO$ preserves the norm, and since $d(x,y)=\n{v}$ by the classic Gauss lemma (cf. \cite[Prop. 6.10]{lee}), we conclude that $d(x,y)=\n{v}=\n{v'}=d(x,y')$.

Finally, let us show that $d_N([x],[y])=d(x,y)=\n{v}$. 
It is clear that $d_N([x],[y])\leq d(x,y)$. 
For the converse, if $(x_0,\dots,x_{2n+1})$ is a chain from $x=x_0$ to $y=x_{2n+1}$, then
$x_{2i}$ and $x_{2i+1}$ are in the same orbit for all $i$, and 
$d(x,x_{2i})=d(x,x_{2i+1})$.
Then 
$$d(x,y)=d(x,x_{2n+1})=d(x,x_{2n})
\leq d(x,x_{2n-1})+d(x_{2n-1},x_{2n})$$
and, by an inductive argument, $d(x,y)\leq \sum_{i=1}^n d(x_{2i-1},x_{2i})$.
Computing the infimum over all the possible chains $(x_0,\dots,x_{2n+1})$ we get $d(x,y)\leq d_N([x],[y])$ and the proof is complete.
%
%
%
%
\end{proof}

\begin{corollary}[cf. \cite{ppt}]
If $(G\toto M,\eta)$ is a proper Riemannian groupoid then $d_N$ is indeed a distance, namely $d_N([x],[y])>0$ if $[x]\neq[y]$.
\end{corollary}


Given $\alpha:I\to [M/G]$ a stacky curve into a Riemannian stack with unit speed, mimicking the classic case, we say that $\alpha$ is {\bf minimizing} if $d_N(\alpha(t),\alpha(s))=\m{t-s}$ for all $s,t\in I$, and $\alpha$ is {\bf locally minimizing} if any $t_0\in I$ has a neighborhood $t_0\in J\subset I$ such that $\alpha|_{J}$ is minimizing. 
Next example shows that stacky geodesics may not be locally minimizing. 


\begin{example}\label{ex:minimizing}
Let $\Z_2$ act over the plane $\R^2$ by the reflection $(x,y)\mapsto(x,-y)$ along the $x$-axis. The canonical metric on $\R^2$ yields a 2-metric on the action groupoid $\Z_2\ltimes\R^2\toto\R^2$, and therefore a metric on the orbit stack $[\R^2/\Z^2]$, that is actually an orbifold (cf. Example \ref{ex:action3}). Then the straight line $t\mapsto (0,t)$ induces a stacky geodesic $\alpha$ on $[\R^2/\Z^2]$ that is not locally minimizing around $0$, for $d_N(\alpha(-\epsilon),\alpha(\epsilon))=0$ for every $\epsilon$.
\end{example}


A heuristic explanation for previous example goes as follows. Geodesics on manifolds are infinitesimal, given by a differential equation, extremals for a variational principle, and since manifolds are locally simply connected, the space of small curves connecting two nearby points is connected, which allows the passing from infinitesimal to global. But stacks are not locally simply connected, there can be fundamental group concentrated on a single point, so a geodesic can have minimum length among all its homotopic curves, and yet not be locally minimizing. 


By the linearization theorem, the objects of a proper Lie groupoid inherit a stratification by isotropy type, and the same holds for the orbit space \cite{ppt}. Next we show that, as suggested in Example \ref{ex:minimizing}, a minimizing geodesic cannot cross different strata (compare with \cite[3.5]{aalm}).

\begin{proposition}\label{prop:stratum}
If $[M/G]$ is a separated stack and $\alpha:I\to[M/G]$ is a stacky minimizing geodesic then the isotropy groups $G_{\alpha(t)}$ are canonically isomorphic, and $\alpha(I)$ is included into a single stratum. 
\end{proposition}

\begin{proof}
Working locally, we can represent $\alpha$ by a normal geodesic $a:J\to M$, and if $x=a(t_0)$, we can replace the original groupoid to its linearization $G_x\times B^{N}_{\epsilon}(x)\toto B^{N}_{\epsilon}(x)$, as done in the proof of Proposition \ref{prop:gauss-lemma}.
We claim that $a(J)\subset B^{N}_{\epsilon}(x)$ remains invariant by the action of the isotropy $G_x$. In fact, given $g\in G_x$, the piecewise smooth curve
$$\tilde a(t)=\begin{cases}
        a(t) & t\leq t_0 \\
        g.a(t) & t\geq t_0
       \end{cases}$$
is also minimizing for $d_N$, then it is also minimizing for $d$, and by classic Riemannian geometry \cite[Thm 6.6]{lee}, it must be a geodesic, and then $a=\tilde a$. This shows that $G_{a(t)}\cong G_x$ for $t>t_0$, with the isomorphism given by the exponential map on the direction of $a'(t_0)$, and similarly for $t<t_0$. The proposition follows by connectedness of $I$.
\end{proof}


In order to provide a characterization of geodesics by means of the normal distance, we introduce the next definition. We say that $\alpha$ is {\bf minimizing at $t_0$} if there exists $\epsilon>0$ such that $d(\alpha(t),\alpha(t_0))=\m{t-t_0}$ whenever $\m{t-t_0}<\epsilon$. A locally minimizing curve is minimizing at every point, and even though the converse does not hold in general (see Example \ref{ex:minimizing}), it does when $[M/G]=N$ is just a manifold \cite[Prop 6.10]{lee}.

\begin{theorem}\label{thm:geodesics-distance}
Let $G\toto M$ be a proper Lie groupoid. Let $\alpha:I\to[M/G]$ be a stacky curve with unit speed. Then $\alpha$ is a stacky geodesic if and only if it is minimizing at every point.
\end{theorem}

\begin{proof}
Suppose $\alpha$ is a geodesic and take $t_0\in I$. By restricting $\alpha$ to a small interval $J=(t_0-\epsilon,t_0+\epsilon)$, we can assume that $\alpha$ is represented by a normal geodesic $a:J\to M$. Write $x=a(t_0)$. If $v=a'(t_0)\in N_{x}M$, then $a(t)=\exp((t-t_0)v)$ and by Gauss lemma \ref{prop:gauss-lemma} we deduce that the following holds and that $\alpha$ is minimizing at $t_0$:
$$d_N(\alpha(t_0),\alpha(t))=d_N([x],[a(t)])=d_N([x],\exp((t-t_0)v))=\m{t-t_0}\n{v}=\m{t-t_0}$$

Conversely, suppose that $\alpha$ is minimizing at every $t\in I$, and let $t_0\in I$. By restricting to a neighborhood of $t_0$ we can assume that $\alpha$ is represented by a groupoid curve $a:J\to M$,
and if $x=a(t_0)$, we can replace the original groupoid to its linearization $G_x\times B^{N}_{\epsilon}(x)\toto B^{N}_{\epsilon}(x)$, as done before.
Then for $t$ near $t_0$ we have
$$\m{t-t_0}=d_N(a(t_0),a(t))=d(x,a(t)),$$
the first identity is because $\alpha$ is minimizing at $t_0$ and the second one is by Gauss lemma \ref{prop:gauss-lemma}. It follows that $a:J\to B^{N}_{\epsilon}(x)$ is minimizing at $t_0$ for both $d$ and $d_N$, then $a$ is both a geodesic and normal, so $\alpha$ is locally a geodesic, and therefore a geodesic (see Lemma \ref{lemma:local-geodesic}).
\end{proof}


\begin{corollary}
Equivalent metrics $\eta_1,\eta_2$ on a proper Riemannian groupoid $G\toto M$ define the same geodesics on the orbit stack. Stacky geodesics $\alpha:I\to [M/G]$ into a separated Riemannian stack only depend on the stacky metric $[\eta]$.
\end{corollary}

\subsection{Complete Riemannian stacks}


Given $(G\toto M,\eta)$ a Riemannian groupoid, we say that $[M/G]$ is {\bf geodesically complete} if any geodesic $\alpha :I\to [M/G]$ can be extended to one $\tilde\alpha: \R \to [M/G]$ defined over the whole $\R$. 
Next we show that if $[M/G]$ is both separated and geodesically complete then the distance between any two points can be realized by a stacky geodesic. The proof mimics the argument used in the manifold case \cite{lee}.

\begin{proposition}\label{prop:geodesic-realizing}
Given $([M/G],[\eta])$ a separated Riemannian stack, if it is geodesically complete and $[x],[y]\in M/G$ are such that $d_N([x],[y])=r$ then there is a unitary stacky geodesic $\alpha:I\to[M/G]$ with $\alpha(0)=x$ and $\alpha(r)=y$.
\end{proposition}

Note that the restriction $\alpha\r{[0,r]}$ is minimizing, but $[x]$, $[y]$ do not need to be in the same stratum, for Proposition \ref{prop:stratum} demands $\alpha$ to be minimizing over the whole open interval.

\begin{proof}
Take $x,y\in M$ representatives for $[x],[y]$, and denote by 
$S^N_{\epsilon}(x)\subset N_xO$ the sphere of normal vectors at $x$ of radius $\epsilon$. By the stacky Gauss Lemma \ref{prop:gauss-lemma} we can take $\epsilon$ such that $d_N([x],[\exp(v)])=\n{v}$ for all $v\in S^N_{\epsilon}(x)$. Without loss of generality we can assume that $\epsilon<r$.
Since $S^N_{\epsilon}(x)$ is compact, there is some $v\in S^N_{\epsilon}(x)$ minimizing the normal distance to $y$, namely
$$d_{N}([\exp(v)],([y])=\inf\{d_{N}([\exp(w)],[y]): \, w\in S^N_{\epsilon}(x)\}.$$

If $z=\exp(v)$ then 
$d_N([z],[y])\geq d_N([x],[y])-d_N([x],[z])=r-\epsilon$
by the triangle inequality, and we claim that the equality holds. Suppose otherwise that $r=d_N([x][y])<d_N([z],[y])+\epsilon$, then by Theorem \ref{thm:distance-length} there is a stacky curve $\alpha:I\to[M/G]$ such that $\alpha(0)=[x]$, $\alpha(1)=[y]$ and $\ell(\alpha)<d_N([z],[y])+\epsilon$. Then if $t_\epsilon$ denotes the time at which $\alpha$ intersects the geodesic sphere $\pi(\exp(S^N_{\epsilon}(x)))\subset M/G$, which exists by connectedness of $I$, we get a contradiction as follows:
$$\ell(\alpha)\geq \ell(\alpha\r{[0,t_\epsilon]})+\ell(\alpha\r{[t_\epsilon,1]})
\geq d_N(\alpha(0),\alpha(t_\epsilon))+d_N(\alpha(t_\epsilon),\alpha(1))\geq \epsilon+d_N([z],[y])$$

Finally, we claim that the unitary stacky geodesic $\alpha$ with initial condition $\alpha(0)=[x]$ and $\alpha'(0)=[\frac{1}{\epsilon}v]$, which by hypothesis is defined for all time, satisfies $\alpha(r)=[y]$, so it minimizes the distance between $[x]$ and $[y]$. 
Let $A=\{t\in [0,r]: d_{N}(\alpha(t),[y])=d_{N}([x],[y])-t\}$ denote the times at which $\alpha$ is optimally approaching $[y]$.
Note that $\epsilon\in A\subset[0,r]$, $t\in A$ for small $t$, and $A$ is closed by continuity. Let $t_0$ be the supremum of $A$. Suppose that $t_0<r$, otherwise we are done. 
By applying the stacky Gauss lemma \ref{prop:gauss-lemma}, and reasoning as before, we know there exists $[w]\in T_{\alpha(t_0)}[M/G]$ such that the stacky geodesic $\beta$ with $\beta(t_0)=\alpha(t_0)$ and $\beta'(0)=[w]$ satisfies 
$d_{N}(\beta(t),[y])=d_N(\beta(0),[y])-t=r-t_0-t$ for small $t$.
To show that we can actually merge $\alpha$ and $\beta$ within a single stacky geodesic, we can pick a representative $z\in M$ of $\alpha(t_0)$, and working locally around $z$, represent $\alpha$ and $\beta$ as normal geodesics $a$ and $b$ in $M$ such that $a(t_0)=b(t_0)=z$. 
Then the piecewise smooth curve 
$$c(t)=\begin{cases}
        a(t) & t\leq t_0 \\
        b(t) & t\geq t_0
       \end{cases}$$
is minimizing for $d_N$, then it is also minimizing for $d$, and by classic Riemannian geometry \cite[Thm 6.6]{lee}, it must be a geodesic, from where $a'(t_0)=b'(t_0)$. This way we can merge $\alpha$ and $\beta$ as in Corollary \ref{cor:merge-geodesics}, contradicting the maximality of $t_0$.
\end{proof}


We can finally present our third main theorem, the following stacky version for classical Hopf-Rinow theorem \cite[Thm 6.13]{lee}.

\begin{theorem}\label{thm:hopf-rinow}
A separated Riemannian stack $([M/G],[\eta])$ is geodesically complete if and only if the coarse orbit space $(M/G,d_{N})$ is a complete metric space.
\end{theorem}

\begin{proof}
Suppose first that $(M/G,d_N)$ is a complete. 
Let $\alpha:I=(t_{min},t_{max})\to [M/G]$ be a stacky geodesic that is maximal, and suppose that $t_{max}<\infty$, the case $-\infty<t_{min}$ is analogous. After a linear reparametrization we may suppose that $\alpha$ is unital. Choose $t_n\nearrow t_{max}$ an increasing convergent sequence. From the inequality 
$$d_N(\alpha(s),\alpha(t))\leq \ell(\alpha|_{[s,t]})=\m{s-t}$$
we see that $\alpha(t_n)$ is a Cauchy sequence, and by the completeness assumption, we have $\alpha(t_n)\to [x]$ for some $x\in M$.
The geodesic flow of the manifold $M$ insures that there exist $\epsilon >0$ and an open $x\in U\subset M$ such that every unitary geodesic $c$ with $c(0)\in U$ is defined at least for time $(-\epsilon, \epsilon)$. Since $\pi:M\to M/G$ is open, we have that $\alpha(t_{n_0})\in \pi(U)$ for some large enough $n_0$, and we can further assume that $t_{max}-t_{n_0}<\frac{\epsilon}{2}$. Choosing $y\in U$ such that $\pi(y)=\alpha(t_{n_{0}})$, and choosing $v\in T_{y}M$ the normal vector representing $\alpha'(t_{n_0})$, we can merge $\alpha$ with the classic geodesic $c$ passing through $y$ with velocity $v$ as in Corollary \ref{cor:merge-geodesics}, contradicting the maximality of $\alpha$. 

For the converse, we assume that the separated Riemannian stack $([M/G],[\eta])$ is geodesically complete. Fix $x\in M$. There is a stacky exponential map
$$\mathcal E:N_xO\to M/G \qquad \exp(v)=\alpha(1)$$
where $\alpha:\R\to [M/G]$ is the stacky geodesic with initial conditions $\alpha(0)=[x]$ and $\alpha'(0)=[v]$. 
It is not hard to see that this map $\mathcal E:N_xO\to M/G$ is continuous. Finally, given a Cauchy sequence in $(M/G,d_N)$, it is a bounded set, and by Proposition \ref{prop:geodesic-realizing}, it sits inside some compact $\mathcal E(B_R^N(x))$. Then the Cauchy sequence must have a convergent subsequent, and therefore be convergent.
\end{proof}


\begin{corollary}
If the coarse orbit space $M/G$ of a separated Riemannian stack $([M/G],[\eta])$ is compact, then the stack is geodesically complete. 
\end{corollary}


We end this section by showing that every separated stack admits a complete stacky metric. This result generalizes the well-known fact that any manifold admits a complete metric.

\begin{corollary}
Every separated stack $[M/G]$ admits a stacky metric $[\eta]$ such that $([M/G],[\eta])$ is geodesically complete.
\end{corollary}

\begin{proof}
Let $G\toto M$ be a proper groupoid let $\eta$ be any groupoid metric on it. 
For each $x\in M$, write $B^\eta_r([x])=\{[y]:d_N([x],[y])<r\}\subset M/G$. 
and let
$R(x)=\sup\{r:\overline{B^\eta_r([x])} \text{ is compact}\}$.
Intuitively, $R(x)$ measures the normal distance between $[x]$ and $\infty$.
If for some $x$ we have $R([x])=\infty$ then any bounded subset of $M/G$ sits inside a compact set, and therefore every Cauchy sequence is convergent, from where the result follows (cf. Theorem \ref{thm:hopf-rinow}). 

Suppose otherwise that $R(x)<\infty$ for all $x$. By the stacky Gauss lemma \ref{prop:gauss-lemma} we have $R(x)>0$, and from the inclusion $B^\eta_r([x])\subset B^\eta_{r+d_N([x],[y])}([y])$ it follows that $\m{R(x)-R(y)}\leq d_N([x],[y])$.
Then $R:M\to \R$ is continuous, we can construct a smooth function $f:M\to\R$ such that $f(x)\geq \frac{1}{R(x)}$ for all $x$, and by an averaging argument, we can make $f$ to be constant along the orbits (cf. \cite{dhf1}). 
We will use $f$ as a conformal factor to build a new metric.

Consider the metric $\tilde f\cdot\eta$ on $G^{(2)}$, where $\tilde f$ is the pullback of $f$ via the structure maps. 
It is easy to see that $\tilde f\cdot\eta$ is a groupoid metric, for $\eta$ is so, and $\tilde f$ is constant along the orbits and invariant under the action $S_3\action G^{(2)}$. If $\alpha:I\to[M/G]$ is a stacky curve, 
$\alpha(t_0)=[x]$, then
$$
\n{\alpha'(t)}_{\tilde f\eta}=
[f](\alpha(t))\n{\alpha'(t)}_{\eta}\geq
\frac{\n{\alpha'(t)}_{\eta}}{[R](\alpha(t))}\geq
\frac{\n{\alpha'(t)}_{\eta}}{R(x)+d_N([x],\alpha(t))}\geq
\frac{\n{\alpha'(t)}_{\eta}}{R(x)+\ell_{\eta}(\alpha)}
$$
and therefore $\ell_{\tilde f\eta}(\alpha)\geq\frac{\ell_{\eta}(\alpha)}{R(x)+\ell_{\eta}(\alpha)}$. By Theorem \ref{thm:distance-length} it follows that
$B^{\tilde f\eta}_{1/3}([x])$ is contained in $\o{B^{\eta}_{R(x)/2}([x])}$, so every Cauchy sequence for $\tilde f\eta$ is contained in a compact set and the result follows.
\end{proof}


\begin{remark}
We remark that, even though every separated stack admits a complete metric, not every proper groupoid $G\toto M$ admits a groupoid metric $\eta$ such that $\eta^{(0)}$ is complete. In fact, there may not exists a transversely invariant complete Riemannian metric $\eta^M$ on $M$, contrary to what was claimed in \cite[Prop. 3.14]{ppt}. For a counterexample, consider the submersion groupoid arising from the first projection $\pi:M=\R^2\setminus\{0\}\to\R$ (cf. Example \ref{ex:submersion1}). Then $\eta^M$ is transversely invariant if and only if $\pi$ becomes a Riemannian submersion (cf. Example \ref{ex:submersion3}), and such an $\eta^M$ cannot be complete because $\pi$ is not locally trivial \cite[Thm 5]{dh0}. In the forthcoming paper \cite{dhdm} we will generalize that result by relating complete groupoid metrics and strict linearization.
\end{remark}


\bigskip
 
\sf{\noindent 
Matias del Hoyo\\
Universidade Federal Fluminense (UFF),\\
Rua Professor Marcos Waldemar de Freitas Reis, s/n
Niteroi, 24.210-201 RJ, Brazil.
\\
mldelhoyo@id.uff.br}

\

\sf{\noindent 
Mateus de Melo\\
Universidade Federal de S{\~a}o Carlos (UFSCar),  \\
Rod. Washington Lu{\'i}s, Km 235, S{\~a}o Carlos, 13.565-905 SP, Brazil.\\
melomm@impa.br}

\end{document}